\documentclass[12pt, a4paper]{amsart}
\usepackage{amsfonts}
\usepackage{bbold}
\usepackage{amssymb}
\usepackage{amsthm}
\usepackage{amsmath}
\usepackage{bbm}
\usepackage{amscd}
\usepackage{stmaryrd}
\usepackage{array}
\usepackage{url}
\usepackage{hyperref}
\usepackage[latin2]{inputenc}
\usepackage{t1enc}
\usepackage{enumerate}
\usepackage[mathscr]{eucal}
\usepackage[british]{babel}
\usepackage{graphicx}
\usepackage{epsfig}
\usepackage{epsf}
\usepackage{color}
\usepackage{indentfirst}
\usepackage[all]{xy}
\usepackage{graphics}
\usepackage{pict2e}
\usepackage{epic}

\newcommand{\Z}{\mathbb{Z}}
\newcommand{\Fp}{\mathbb{F}_p}
\newcommand{\Qp}{\mathbb{Q}_p}

\newcommand{\sgn}{\operatorname{sgn}}

\newtheorem{thm}{Theorem}

\newtheorem{constr}[thm]{Construction}
\newtheorem{lem}[thm]{Lemma}

\theoremstyle{definition}
\newtheorem*{rem}{Remark}

\begin{document}
\title[Estimating the gcd of the value of two polynomials]{Estimating the greatest common divisor of the value of two polynomials}
\author{P\'eter E.\ Frenkel
 \and
  Gergely Z\'abr\'adi}
\address{ELTE E\"{o}tv\"{o}s Lor\'{a}nd University \\ Faculty of Science \\  Institute of Mathematics\\ 1117 Budapest, Hungary
\\ P\'{a}zm\'{a}ny P\'{e}ter s\'{e}t\'{a}ny 1/C \&  Alfr\'ed R\'enyi Institute of Mathematics, Hungarian Academy of Sciences \\
1053 Budapest, Hungary \\
Re\'altanoda u.\ 13-15.  ORCID ID: 0000-0003-2672-8772 and 0000-0002-7293-3569}
\email{frenkelp@cs.elte.hu, zger@cs.elte.hu}
\thanks{This project has received funding from  the European Research Council
(ERC) under the European Union's Horizon 2020 research and innovation
program (grant agreement No.\ 648017), from the MTA R\'enyi Lend\"ulet
Groups and Graphs research group, from the Hungarian National
Research, Development and Innovation Office -- NKFIH, OTKA grants no.\
K104206 and K109684, and from the J\'anos Bolyai Scholarship of the Hungarian Academy of Sciences.}
\date{\today}

\begin{abstract} Let $p$ be a fixed prime, and let $v(a)$ stand for the exponent of $p$ in the prime factorization of the integer $a$.
Let $f$ and $g$ be two monic polynomials with integer coefficients and  nonzero resultant $r$.
Write 
 $S$ for the 
  maximum of $v(\gcd (f(n), g(n)))$ over all integers  $n$.  It is known that $S  \le v(r)$.  We give various lower and upper bounds for   the least possible value of $v(r)-S$ provided that a  given power $p^s$ divides both $f(n)$ and $g(n)$ for all $n$. In particular, the least possible value is
$ps^2-s$  for $s\le p$ and is asymptotically $(p-1)s^2$ for large $s$.
\end{abstract}
\maketitle
Let $f,g\in\mathbb{Z}[x]$ be monic polynomials with nonzero resultant $r$.  Our interest is in the range of the greatest common divisor of $f(n)$ and $g(n)$ as $n$ varies in $\Z$.  In the recent paper \cite{FP} by J.\  Pelik\'an and the first author, it was shown\footnote{Statement~\eqref{high} was essentially known before, cf.~\cite{GGIS, Kh}}  that
\begin{enumerate}
\item $\gcd(f(n),g(n))$ divides  $r$ for all $n$; moreover,

\item\label{squarefree} for square-free $r$, its range is the set of all (positive) divisors of $r$;

\item\label{notsquarefree} If  $r$ is allowed to  have square divisors, then $|r|$ need not be in the range.  For example, $f(x)=x^2+1$ and $g(x)=x^2-1$
have resultant 4 but  never have gcd 4.

\item\label{high} If $r$ has no divisors of the form $p^p$  with $p$ prime, then 1 appears in the range.
\end{enumerate}

For  statement~\eqref{notsquarefree}, there is an even worse example with resultant 4: $f(x)=
x^2+x+1$ and $g(x)=x^2+x-1$  have $f(n)$ and $g(n)$ coprime for all $n$.  For statement~\eqref{high} with the condition on $r$ removed, there again is a  counterexample with resultant 4:
$f(x)=
x^2+x+2$ and $g(x)=x^2+x$   have $\gcd(f(n),g(n))= 2$ for all $n$.   On the other hand, it will turn out that if $r$ is in the range, then so are all its divisors; see  Theorem~\ref{main}
 below.

In the present paper, we undertake  a  refined study of the case when $r$ can have   prime power  divisors with high exponents. Fix  a  prime $p$, and let $v(a)$ stand for the exponent of $p$ in the prime factorization of the integer $a$. It suffices to study the range of $v(\gcd(f(n),g(n)))$, since if we understand this for all $p$, then the Chinese remainder theorem allows us to read off the range of
$\gcd(f(n),g(n))$.

 Write 
  $S$ for the 
   maximum of $v(\gcd (f(n), g(n)))$ as $n$ varies in $\Z$.  By \cite[Proposition 2(a)]{FP}, we have $S  \le v(r)$.
 Our main goal is to estimate
   the least possible value of $v(r)-S$ provided that 
   $v(\gcd(f(n),g(n)))\ge s$ for all $n$.   We develop two different methods.
 Up to Theorem~\ref{mainlarge},  we
 use the definition of the resultant in terms of the coefficients of $f$ and $g$,
 while from Construction~\ref{sharplarge} on, we 
 use the equivalent definition  in terms of the roots of $f$ and $g$.

 Let
 \begin{equation}\label{f}f(x)=a_0x^k+a_1x^{k-1}+\dots +a_k\end{equation} and
  \begin{equation}\label{g}g(x)=b_0x^l+b_1x^{l-1}+\dots +b_l,\end{equation} where $a_0=b_0=1$.
 Recall that, by definition, $r$ is the determinant of the \it Sylvester matrix \rm \begin{equation}\label{Sylv}M=\begin{pmatrix} a_0 & a_1 & \dots& \dots & a_k & & \\&a_0 & a_1 & \dots&\dots  & a_k &\\  & & \dots&\dots  & \dots &\dots &\dots \\&&&a_0 & a_1 & \dots &\dots & a_k \\ b_0 & b_1 & \dots &\dots & b_l & & \\&b_0 & b_1 & \dots &\dots & b_l &\\  & & \dots&\dots   & \dots &\dots &\dots \\&&&b_0 & b_1 & \dots &\dots & b_l
\end{pmatrix}\end{equation} of the two polynomials. Note that $M$ is an $(l+k)$-square matrix; the first $l$ rows are built from the coefficients of $f$, and the last $k$ rows
 are built from the coefficients of $g$, padded with zeros.

We shall need the following interpretation of the resultant.

  \begin{lem}\label{quotientring} If $f$ and $g$ are monic polynomials with integer coefficients and nonzero resultant $r$, then $|r|=|\Z[x]/(f,g)|$, where $(f,g)$ stands for the ideal generated by $f$ and $g$.
  \end{lem}

Note that for $r=0$  (which is excluded throughout this paper), we would have $ |\Z/(f,g)|=\infty$ because $f$ and $g$ would have  a nonconstant common divisor in $\Z[x]$.

Note also that Lemma~\ref{quotientring} implies  \cite[Proposition 2(a)]{FP}: the greatest common divisor $(f(n),g(n))$ divides the resultant $r$.  Indeed, there is a surjective ring homomorphism from
$\Z[x]/(f,g)$ onto $\Z/(f(n),g(n))$.

   The statement and proof of Lemma~\ref{quotientring}  are reminiscent of \cite[Theorem 1.19]{J}, which was reproved as \cite[Theorem 5]{FP}.  In that theorem, the coefficients come from a  field $F$, and the  claim is that the corank of the Sylvester matrix  $M$ is the dimension over $F$ of the quotient  ring $F[x]/(f,g)$, i.e., the degree of the polynomial $\gcd(f,g)$.

  \begin{proof} Let us identify the free Abelian group $\Z^{k+l}$ with the additive group $\Z[x]_{<k+l}$ of polynomials of degree less than $k+l$ with integer coefficients.  Let any such polynomial correspond to the list of its coefficients, starting with the coefficient of $x^{k+l-1}$ and ending with the constant term.

  Under this correspondence, the subgroup generated by the rows of the Sylvester matrix $M$ is identified with the set of polynomials of the form $\phi f+\psi g$, where $\phi, \psi\in \Z[x]$ have degree less than $l$ and $k$, respectively. Any polynomial of this form is in $(f,g)$. Conversely, any element of  $(f,g)$ of degree less than $k+l$ is an integral linear combination of the rows. To see this, we first write such a polynomial as 
  $\phi_0 f+\psi_0 g$, where we know nothing about the degree of $\phi_0, \psi_0\in \Z[x]$, but then we write $\phi_0=qg+\phi$ with $\phi$ of degree less than $l$, and we define $\psi=qf+\psi_0$.  Then $\phi_0 f+\psi_0 g=\phi f+\psi g$; moreover, this polynomial and $\phi f$ both have degree less than $k+l$, whence so does $\psi g$, showing that $\psi$ has degree less than $k$.

  Thus, the subgroup of $\Z^{k+l}$ generated by the rows of $M$ is identified with the degree $< k+l$ part $(f,g)_{<k+l}$ of the ideal $(f,g)$ of $\Z[x]$. The determinant $r$  of $M$ is the signed volume of the parallelotope spanned by  the rows, therefore $|r|$ is the volume of this parallelotope, which is the cardinality of the quotient  \begin{align*}\Z^{k+l}/\langle\textrm{rows of }M\rangle \simeq\Z[x]_{<k+l} & /(f,g)_{<k+l}\simeq \\ \simeq((f,g)+\Z[x]_{<k+l}) & /(f,g)=\Z[x]/(f,g).\end{align*}
      \end{proof}

For integers $S\ge s\ge 0$, let $$I_{S,s}=\left\{f\in\Z[x]:  p^s|f(n) \textrm{ for all } n, \textrm{ and }  p^S|f(0)\right\}.$$  This is an ideal of $\Z[x]$. Put $R_{S,s}=\Z[x]/I_{S,s}$. The cardinality of this quotient ring will play  a central role in our computations.
The cardinality can be expressed  in terms of the functions 
 $$\alpha (j)=v(j!)=\left\lfloor \frac jp\right\rfloor+\left\lfloor \frac j{p^2}\right\rfloor+\left\lfloor \frac j{p^3}\right\rfloor+\dots$$ and $\beta  (m)=\min\{j:\alpha(j)\ge m\}$.
Put $B(s)=\sum_{m=1}^s\beta(m)$.

Note that $\alpha$ is superadditive: $$\alpha(j_1+j_2)\ge\alpha(j_1)+\alpha(j_2)$$ for all nonnegative integers $j_1$ and $j_2$.  It follows that $\beta$ is subadditive:
$$\beta(m_1+m_2)\le\beta(m_1)+\beta(m_2)$$ for all nonnegative integers $m_1$ and $m_2$.

Note also that $\alpha(j)=\lfloor j/p\rfloor$ for $0\le j< p^2$, and $\alpha(p^2)=p+1$, whence $\beta(m)=pm$ for $1\le m\le p$ and $B(s)=p\binom{s+1}2$ for $1\le s\le p$.  On the other hand, $\alpha (j)\sim j/(p-1)$ for large $j$, whence $\beta (m)\sim (p-1)m$ for large $m$ and $B(s)\sim (p-1)s^2/2$ for large $s$.

\begin{lem}\label{kellerolson}
We have $$|R_{S,s}|=p^{S-s+B(s)}.$$
\end{lem}

\begin{proof}
For $S=s$,  the ring $R_{S,s}=R_{s,s}$ is the  ring of polynomial functions $\Z/(p^s)\to
\Z/(p^s)$. By a classical  result  of Kempner~\cite{K}, reproved by Keller and Olson~\cite[Corollary 2.2]{KO}, this ring  has cardinality $p^{B(s)}$. 

For $S\ge s$, observe that $I_{S,s}$ is the kernel of the map   $I_{s,s}\to \Z/(p^S)$, $f\mapsto f(0)$. The image of this map is $(p^s)/(p^S)$, whence  $|I_{s,s}/I_{S,s}|=p^{S-s}$. But $I_{s,s}/I_{S,s}$ is the kernel of the surjective map $R_{S,s}\to R_{s,s}$, therefore $| R_{S,s}|/| R_{s,s}|=p^{S-s}$ and the Lemma follows.
\end{proof}

The first  main result of this paper is the following refinement of  \cite[Proposition 8(a)]{FP}.
\begin{thm}\label{mainlarge}
  Let $f$ and $g$ be monic polynomials with integer coefficients and nonzero resultant $r$. Assume that  a  fixed prime power $p^s$ divides both $f(n)$ and $g(n)$ for all $n$. Let $$
  S=\max_{n\in\Z}v(\gcd(f(n),g(n))).$$ Then $v(r)-S\ge B(s+t)-2B(t)-s$  for all nonnegative integers $t$.
\end{thm}

\begin{proof} The resultant being translation invariant, we may and do assume that $p^S$ divides $\gcd(f(0),g(0))$. Using Lemma~\ref{quotientring}, we have \begin{align*}v(r)=&v\left(|\Z[x]  /(f,g)|\right)\ge \\ \ge &v\left(|\Z[x]/((f,g)+I_{S+t, s+t})|\right)=v\left(\left|R_{S+t, s+t}/\left(\bar f, \bar g\right)\right|\right),\end{align*}
where $\bar f$ and $\bar g$ are the natural images in
$R_{S+t, s+t}$  of $f$ and $g$, respectively.  Now observe that in the $\Z[x]$-module $R_{S+t, s+t}$, both elements $\bar f$ and $\bar g$ are  annihilated by the ideal $I_{t,t}$.  Hence $v\left(\left|\left(\bar f\right)\right|\right)\le v(|R_{t,t}|)=B(t)$  by Lemma~\ref{kellerolson}, and similarly for $\bar g$.   Now $$v\left(\left|\left(\bar f, \bar g\right)\right|\right)=v\left(\left|\left(\bar f\right)\right|\right)+v(|( \bar g)|)-v\left(\left|\left(\bar f\right)\cap (\bar g)\right|\right)\le 2B(t),$$  whence
$$v(r)\ge v(|R_{S+t, s+t}|)-v\left(\left|\left(\bar f, \bar g\right)\right|\right)\ge (S+t)-(s+t)+B(s+t)-2B(t)$$ and the Theorem follows.
\end{proof}

For $s=1$, we may choose $t=0$ in Theorem~\ref{mainlarge} to get $v(r)\ge  S+p-1\ge p$, which recovers \cite[Proposition 8(a)]{FP}.
For general $s\ge 0$, choosing $t=s$, we get
$v(r)-S\ge B(2s)-2B(s)-s$. When $s\le p/2$, we have  $B(s)=p\binom {s+1}2$ and $B(2s)=p\binom {2s+1}2$, whence $v(r)-S\ge ps^2-s$. It shall follow from Theorem~\ref{main} and Construction~\ref{sharpsmall} that this lower bound holds true, and is sharp, even
under the weaker assumption that $s\le p$. On the other hand, for large $s$, we have $B(s)\sim (p-1)s^2/2$ and
$B(2s)\sim 2(p-1)s^2$, whence $v(r)-S\gtrsim (p-1)s^2$. We now present a construction showing 
 that this is asymptotically sharp for any fixed $p$.

 \begin{constr}\label{sharplarge}
Consider the polynomials
\begin{eqnarray*}
f(x)&:=&\prod_{j=0}^{\beta(s)-1}(x-j)\ ;\\
g(x)&:=&p^s+\prod_{i=0}^{p-1}(x-i)^{s+1}\end{eqnarray*}  for  an integer $s\ge 0$.
Then 
$v(\gcd(f(n),g(n)))=s$ for all integers $n$.  
 For the resultant $r$, we have $v(r)=s\beta(s)$, whence
 $v(r)-s=s(\beta(s)-1)\sim (p-1)s^2$ 
 when $s\gg p$.
\end{constr}

\begin{proof}
Firstly, note that $f(\beta(s))=\beta(s)!$ divides $ f(n)$ for any integer $n$ since the binomial coefficient $\binom{n}{\beta(s)}={f(n)}/{\beta(s)!}$ is an integer. Therefore, we have $s\le \alpha(\beta (s))= v(\beta(s)!)\leq v(f(n))$. On the other hand, we have $v(g(n))=s$ for all $n$ since $p^{s+1}$ divides $ \prod_{i=0}^{p-1}(n-i)^{s+1}$ for any integer $n$. Hence the statement on 
$v(\gcd(f(n),g(n)))$. Further, we compute
\begin{align*}
v(r)=v\left(\prod_{j=0}^{\beta(s)-1}g(j)\right)=\sum_{j=0}^{\beta(s)-1}v(g(j))=s\beta(s)\ .
\end{align*}
\end{proof}

Let us return to the notations and conditions of Theorem~\ref{mainlarge}.
In the rest of this paper, our main goal is to obtain a  sharp lower bound for $v(r)-S$ when $s\le p$. For this, we recall  a bit of $p$-adic number theory. 
Let $K$ be the splitting field of the product $fg$ over the field $\Qp$ of $p$-adic numbers for the fixed  prime $p$. So we may write $f(x)=\prod_{i=1}^k(x-\gamma_i)$ and $g(x)=\prod_{j=1}^l(x-\delta_j)$ with $\gamma_i,\delta_j\in \mathcal{O}$ ($i=1,\dots,k$; $j=1,\dots,l$), where $\mathcal{O}$ denotes the valuation ring in $K$ with uniformizer $\pi$ and residue field $\mathbb{F}= \mathcal{O}/(\pi)$. We put $e=v_\pi(p)$ for the absolute ramification index of $K$, 
 where $v_\pi$ stands for the $\pi$-adic valuation. We extend the $p$-adic valuation $v$ to $K$ by putting $v=v_\pi/e$. In particular, we have $v(\pi)={1}/{e}$, and  the $v$-value of any element of $\mathcal O$ is  a nonnegative integer multiple of $1/e$.
We have $e\cdot |\mathbb F: \Fp|=|K:\Qp|$, but this will not be used in the sequel.





For integers $n\in\mathbb{Z}$ and $0\le s\in\mathbb{Z}
$, the value $f(n)\in\mathbb{Z}$ is divisible by $p^s$ if and only if $\sum_{i=1}^kv(n-\gamma_i)\geq s$. On the other hand, the resultant of $f$ and $g$ equals $$r=
\prod_{i,j}(\gamma_i-\delta_j)\in\mathbb{Z}.$$ For any fixed $n\in \mathbb{Z}$, we have the following trivial estimate for the $p$-adic valuation of $r$:
\begin{align}
v(r)=\sum_{i,j}v(\gamma_i-\delta_j)\geq \sum_{i,j}\min(v(n-\gamma_i),v(n-\delta_j))\ .\label{trivialvpRu}
\end{align}

Note that the above trivial estimate again implies  \cite[Proposition 2(a)]{FP}: the greatest common divisor $(f(n),g(n))$ divides the resultant $r$. 
Indeed, it suffices to check this locally, i.e.,  \begin{align*}v(\gcd(f(n),g(n)))=\min(v(f(n)),v(g(n)))=\\=\min\left(\sum_{i}v(n-\gamma_i),\sum_{j}v(n-\delta_j)\right)\leq v(r)\end{align*} for all primes $p$. The latter inequality follows easily from \eqref{trivialvpRu} by choosing a maximum among the multiset $$\{v(n-\gamma_i),v(n-\delta_j)\mid 1\leq i\leq k,1\leq j\leq l\}.$$

In order to estimate this further from below, we need the following lemma stating (in the special case of $I=\emptyset$) that whenever $s\le p$ and $f(n)$ is divisible by $p^s$ for all $n$, then there are at least $s$ roots of $f$ in $\overline{\mathbb{Q}_p}$ congruent to each integer modulo $p$.

\begin{lem}\label{partialmanycongruent}
Let $m\in \mathbb{Z}$ be a fixed integer, and let $I\subseteq \{1,\dots,k\}$ be an arbitrary subset such that for all $i\in I$ we have $v(m-\gamma_i)\notin\mathbb{Z}$. Further, let $0\leq t_I<p$ be the number of indices $i\in \{1,\dots,k\}\setminus I$ with $v(m-\gamma_i)>0$. Then there exists an integer $n\in\mathbb{Z}$ such that $n\equiv m\pmod{p}$ and $v(f(n))\leq \sum_{i\in I}v(m-\gamma_i)+ {t_I}$.
\end{lem}
\begin{proof}
First of all, note that $$v(f(n))=\sum_{i=1}^kv(n-\gamma_i)= \sum_{i\in I}v(n-\gamma_i)+\sum_{i\in \{1,\dots,k\}\setminus I}v(n-\gamma_i).$$

On the one hand, for any integer $n\in\mathbb{Z}$ and $i\in I$,  we have $v(n-m)\in\mathbb{Z}$, whence $v(n-m)\neq v(m-\gamma_i)$, as the latter is not an integer by assumption. So we compute $$v(n-\gamma_i)=v((n-m)+(m-\gamma_i))=\min(v(n-m),v(m-\gamma_i))\leq v(m-\gamma_i).$$

On the other hand, we want  to pick $n\in\mathbb{Z}$ in such a way that we can estimate $$\sum_{i\in \{1,\dots,k\}\setminus I}v(n-\gamma_i)$$ efficiently. 
We have to have $n\equiv m\pmod{p}$, and we choose $n$ modulo $p^2$ so that 
 all indices $i\in \{1,\dots,k\}\setminus I$  satisfy $v(n-\gamma_i)\le1$ (equivalently, $ < 1+{1}/{e}$). Indeed, we can achieve this by the pigeonhole principle: there are $p$ choices for $n$ mod $p^2$ and these are pairwise incongruent mod $\pi^{e+1}$, so any element $\gamma\in\mathcal{O}$ can only be congruent to one of these choices modulo $\pi^{e+1}$.

 This way we obtain an integer $n\equiv m\pmod{p}$ such that
\begin{align*}
v(f(n))= & \sum_{i=1}^kv(n-\gamma_i) \leq \\ \leq & \sum_{i\in I}v(m-\gamma_i) + \sum_{i\not\in 
I, v(m-\gamma_i)>0}1
=  \sum_{i\in I}v(m-\gamma_i) + t_I
\end{align*}
as desired.
\end{proof}




The  second main result of this paper is the following refinement of  \cite[Proposition 8(a)]{FP}.

\begin{thm}\label{main}  Let $f$ and $g$ be monic polynomials with integer coefficients and nonzero resultant $r$. Assume that $s\le p$ and that the power $p^s$ divides both $f(n)$ and $g(n)$ for all $n$.
Let $$
S=\max_{n\in\mathbb{Z}}v(\gcd(f(n),g(n))).$$
\begin{enumerate}
\item[(a)] 
We have $$v(r)-S\ge ps^2-s\ .$$

\item[(b)] If equality
 holds here, then $v(\gcd(f(n),g(n)))$ takes all the integer values in the interval $[s,S]$.
 \end{enumerate}
\end{thm}

\begin{proof}
We may assume without loss of generality that $$v(\gcd(f(0),g(0)))=S.$$ Fix an integer $m\in\Z$, and set $a_i=v(m-\gamma_i)$ and $b_j=v(m-\delta_j)$ ($i=1,\dots,k$;  $j=1,\dots,l$). 
By assumption, $p^s$ divides $\gcd(f(m),g(m))$, so we have $\sum_{i=1}^ka_i\geq s$ and $\sum_{j=1}^lb_j\geq s$.

 We may assume without loss of generality (possibly swapping $f$ and $g$ and permuting their roots) that the maximum of $$\{a_i,b_j\mid 1\leq i\leq k,1\leq j\leq l\}$$ is achieved at $b_l$.
\begin{lem}\label{fixedvestimate}
\begin{enumerate}
\item[(a)] 
We have
$$\sum_{i,j\colon \gamma_i\equiv m\equiv\delta_j\pmod{\pi}}v(\gamma_i-\delta_j)\geq \begin{cases}s^2\qquad &(m\in\Z)\\ s^2-s+S\quad &(m=0).\end{cases}$$
\item[(b)] If equality holds for    $m=0$, then either $S=s$, or all of the following hold: 
    $$b_l\ge
    S-s+\sgn s,$$
    $$b_j\le \sgn s \textrm{   
     for all } j<l,$$ and  $$\sum_{j=1}^{l-1}b_j=
     s-\sgn s.$$
\end{enumerate}
\end{lem}
Here $\sgn 0=0$ and $\sgn s =1$ for $s\ge 1$.
\begin{proof}
(a) 
We have $v(\gamma_i-\delta_j)\geq \min(a_i,b_j)$ as before. Note that whenever $m\not\equiv\gamma_i\pmod{\pi}$ or $m\not\equiv\delta_j\pmod{\pi}$, then $\min(a_i,b_j)$ vanishes. Hence we obtain
\begin{align}
\sum_{i,j\colon \gamma_i\equiv m\equiv\delta_j\pmod{\pi}}v(\gamma_i-\delta_j)\geq \sum_{j=1}^l\sum_{i=1}^k\min(a_i,b_j)\label{firstandtrivial}
\end{align}
by adding these together. Fix $j\in\{1,\dots,l\}$ for now, and put $$I_j:=\{i\in\{1,\dots,k\}\mid a_i\leq \min(1,b_j)\text{ and }a_i\notin\mathbb{Z}\}\ .$$ Let $t_j$ be the number of indices $i\in \{1,\dots,k\}\setminus I_j$ such that $a_i\neq 0$. 
Applying Lemma~\ref{partialmanycongruent} 
 to the subset $I:=I_j$, we find $$s\leq \sum_{i\in I_j}a_i+t_j\ .$$ On the other hand, for any $i\in \{1,\dots,k\}\setminus I_j$ with $a_i\neq 0$, we have $a_i\geq \min(1,b_j)$, so 
\begin{equation}\label{sumfixedj}
\begin{aligned}
\sum_{i=1}^k\min(a_i,b_j)\geq \sum_{i\in I_j}a_i+t_j\min(1,b_j)
\geq\\\geq\left(\sum_{i\in I_j}a_i+t_j\right)\min(1,b_j)
\geq
s\min(1,b_j)\ .
\end{aligned}\end{equation}
Now Lemma \ref{partialmanycongruent} 
 applied to the polynomial $g$ and to the subset $$I:=\{j\in\{1,\dots,n\}\mid 0<b_j<1\}$$ yields
\begin{align}\label{summinbj1}
s\le\sum_{j\in I}b_j
+t_I\le
\sum_{j=1}^l\min(1,b_j).
\end{align}
The first statement in (a) is a combination of \eqref{firstandtrivial}, \eqref{sumfixedj}, and \eqref{summinbj1}.


Let $m=0$. By the maximality of $b_l$, we have
\begin{align*}
\sum_{i=1}^k\min(a_i,b_l)=\sum_{i=1}^ka_i=v(f(0)) \geq  S
\ .
\end{align*}
Also, 
\begin{align*}
1+\sum_{j=1}^{l-1}\min(1,b_j)\geq \sum_{j=1}^l\min(1,b_j)\geq s\ .
\end{align*}
This yields
\begin{align*}
\sum_{j=1}^l\sum_{i=1}^k\min(a_i,b_j)=\sum_{j=1}^{l-1}\sum_{i=1}^k\min(a_i,b_j)+\sum_{i=1}^k\min(a_i,b_l)
\geq \\ \ge s\sum_{j=1}^{l-1}\min(1,b_j)+S
\geq s(s-1)+S
\end{align*}
as desired.

(b) Fix $j<l$. To have equality in the last chain of inequalities, we must have equality in \eqref{sumfixedj}, whence $\min(a_i,b_j)=\min (1, b_j)$  for all $i$ such that $i\not\in I_j$ and $a_i>0$. 
 We must also have $\sum_{i=1}^ka_i=S$ and, in case $s\ge 1$, we must have $b_l\ge 1$ and
 $\sum_{j=1}^l\min(1,b_j)=s$.

If $b_j>1$ for some $j<l$, then $a_i=1$  for all $i$ such that $i\not\in I_j$ and $a_i>0$, which means that $a_i\le 1$ for all $i$. But \eqref{sumfixedj} holds with equality, so we have $\sum_{i=1}^ka_i=s$, whence $S=s$.


If $b_j\le 1$ for all $j<l$, then $\min (1, b_j)=b_j$ for all $j<l$, hence $$S\le v(g(0))=\sum_{j=1}^lb_j=b_l+\sum_{j=1}^{l-1}\min(1, b_j).$$  If $s\ge 1$, then this is $b_l-1+s$, and $b_l\ge S-s+1$ follows.  If $s=0$, then, since \eqref{sumfixedj} holds with equality, we deduce either $a_1=\dots=a_k=0$ and therefore $S=0=s$, or $b_1=\dots=b_{l-1}=0$ and therefore $b_l\ge S$.
\end{proof}
Adding up the estimates of 
 Lemma~\ref{fixedvestimate}(a) for $m=0, 1,\dots,p-1$, we deduce Theorem~\ref{main}(a).  For (b), observe that the value $S$ is obviously taken. Observe also that if $v(r)-S=ps^2-s$, then
 the value $s$ is also taken, for otherwise Theorem~\ref{main}(a)  yields $$v(
 r)-S\ge p(s+1)^2-(s+1),$$  a contradiction.  Moreover,
 equality holds in Lemma~\ref{fixedvestimate}(a) for all $m$, in particular, for $m=0$.  Thus, Lemma~\ref{fixedvestimate}(b) applies. If $S=s$, then Theorem~\ref{main}(b)  obviously holds.  We treat the other case  given in Lemma~\ref{fixedvestimate}(b).
Let  $\sgn s< u< S-s+\sgn s$. 

 We have $v(p^u-\delta_l)=u$ and $v(p^u-\delta_j)=b_j$ for all $1\leq j\leq l-1$. So we compute $$v(g(p^u))=\sum_{j=1}^lv(p^u-\delta_j)=u+\sum_{j=1}^{l-1}b_j=
 u+s-\sgn s.$$

 We have $v(f(p^u))\geq u$, but also $v(f(p^u))\geq s+u-1$.  To prove the latter, we distinguish two cases. If $a_i\leq u$ for all $1\leq i\leq k$, then $v(p^u-\gamma_i)\geq a_i$, which yields $$v(f(p^u))=\sum_{i=1}^k v(p^u-\gamma_i)\geq \sum_{i=1}^k a_i\geq S>s+u-1.$$ So assume that there exists an index $1\leq i\leq k$ with $a_i>u$, say  $a_k>u$. Put $$I:=\{1\leq i\leq k\mid 0< a_i<1\}$$ and let $t_I$ be the number of indices $i$ 
 with $a_i\ge 1$. By Lemma~\ref{partialmanycongruent}, we find $ s\le\sum_{i\in I}a_i+t_I$. 
 On the other hand, we have $v(p^u-\gamma_i)=a_i$ for all $i\in I$. 
 Summing 
 yields
\begin{align*}
v(f(p^u))=\sum_{i=1}^k v(p^u-\gamma_i)=u+\sum_{i=1}^{k-1} v(p^u-\gamma_i)=\\
=u+\sum_{i\in I}a_i+\sum_{i\in \{1,\dots,k-1\}\setminus I}v(p^u-\gamma_i)\geq u+\sum_{i\in I}a_i+t_I-1\geq
u+s-1.
\end{align*}

We deduce that $$v(\gcd(f(p^u),g(p^u)))=
u+s-\sgn s,$$
which takes all integer  values in the open interval $(s,S)$ when $u$ runs
over integers in $(\sgn s,S-s+\sgn s)$. 
\end{proof}

\begin{rem}
Assuming $s\ge1$ and noting $S\geq s$ in Theorem~\ref{main}(a) yields $v(r)\geq p$, which is the statement of  \cite[Proposition 8(a)]{FP}.
\end{rem}

\begin{rem}
The above proof shows that one can weaken the assumption in Theorem~\ref{main}(b): it suffices to assume that the estimate in case $m=0$ of Lemma~\ref{fixedvestimate}(a) is sharp for the choice of $f$ and $g$.
\end{rem}

\begin{constr}\label{sharpsmall}
Let $p$ be a prime and assume that $0\leq s\leq S$ and, in case $p=2\leq s$, also that $2s+1\leq S$. Then there exists a pair $f,g\in\mathbb{Z}[x]$ of monic polynomials such that  $\min_{n\in\mathbb{Z}}v(\gcd(f(n),g(n)))=s$, $\max_{n\in\mathbb{Z}}v(\gcd(f(n),g(n)))=S$, and $v(r)-S=ps^2-s$  holds for the resultant $r$.  In particular, the estimate in Theorem~\ref{main}(a) is sharp  for any prime $p\geq 2$ and any $0\le s\le p$.
\end{constr}
\begin{proof}
If $s=S=0$ we simply take $f(x)=1$ and $g(x)$ arbitrary. In case $s=0<S$ (resp.\ $s=1\leq S$) we pick $f(x)=x$ (resp.\ $f(x)=x(x-1)$) and $g(x)=x-p^S$ (resp.\ $g(x)=(x-p^S)(x-1-p)$).

For $s\geq 2$ and $p$ odd, the example is $$f(x)=x(x-2p)^{s-1}\prod_{j=1}^{p-1}(x-j)^s$$ and $$g(x)=(x-p^{S-s+1})(x-p)^{s-1}\prod_{j=1}^{p-1}(x-j-p)^s\ .$$ Under this choice, we clearly have $s=\min_{n\in\mathbb{Z}}v(\gcd(f(n),g(n)))$. On the other hand, $f(0)=0$ and $v(g(0))=S$, whence $$\max_{n\in\mathbb{Z}}v(\gcd(f(n),g(n)))\leq S.$$ Moreover, if $n\equiv j\neq 0\pmod{p}$ ($j=1,\dots,p-1$),  then $n$ cannot be congruent to both $j$ and $j+p$ modulo $p^2$, whence $$v(\gcd(f(n),g(n)))=s.$$ Further, if $p\mid n$, then we distinguish three cases:
\begin{enumerate}[$(i)$]
\item $n\equiv 0\pmod{p^{S-s+2}}$. Then $$v(n-p^{S-s+1})=S-s+1 \textrm{ and } v(n-p)=1,$$ whence $v(g(n))=S$.
\item $n\equiv p^{S-s+1}\pmod{p^{S-s+2}}$. Then $$v(n)=S-s+1 \textrm{ and } v(n-2p)=1,$$ showing that $v(f(n))=S$.
\item $0\not\equiv n\not\equiv p^{S-s+1}\pmod{p^{S-s+2}}$. In this case, we have $$v(n)=v(n-p^{S-s+1})\leq S-s+1,$$ and $n$ cannot be congruent to both $p$ and $2p$ modulo $p^2$, showing that $v(\gcd(f(n),g(n)))\leq S$.
\end{enumerate}
In all cases, we obtained $v(\gcd(f(n),g(n)))\leq S$, showing that  $S$ is the maximum. Finally, we compute
\begin{align*}
v(r)= v\left(g(0)g(2p)^{s-1}\prod_{j=1}^{p-1}g(j)^s\right)=\\=v(g(0))+(s-1)v(g(2p))+s\sum_{j=1}^{p-1}v(g(j))=\\
=S+(s-1)s+s(p-1)s=ps^2-s+S
\end{align*}
as claimed.

Finally, if $p=2\leq s\leq ({S-1})/{2}$, then we take $$f(x)=x(x-2)^{s-1}(x-1)^s$$ and $$g(x)=(x-2^{S-2s+2})(x-4)^{s-1}(x-3)^s.$$ A simple computation similar to the one above shows the statement.
\end{proof}



\begin{thebibliography}{99}
\bibitem{FP} Frenkel P.\ E., Pelik\'an J., On the Greatest Common Divisor of the Value of Two Polynomials, \emph{The American Mathematical Monthly}
\textbf{124}(5) (May 2017), 446--450.

\bibitem{GGIS} D.\ Gomez, J.\ Gutierrez, \'A.\  Ibeas,  D.\ Sevilla, Common factors of resultants modulo $p$, \it
Bull.\ Aust.\ Math.\ Soc.
\rm {\bf 79}  (2009),  299--302.


\bibitem{J} S.\ Janson, Resultant and discriminant of polynomials,
\newline
{\tt{https://www.semanticscholar.org}}, 2010 

\bibitem{KO}  Keller G., Olson F.\ R., Counting polynomial functions (mod $p^n$),  \emph{Duke Math.\ J.} \textbf{35} (1968), 835--838.

    \bibitem{K}  Kempner A.\ J., Polynomials and their residue systems, \emph{Amer.\ Math.\ Soc.\ Trans.} \textbf{22}  (1921), 240--288.

        \bibitem{Kh} D.\ I.\ Khomovsky,  On the relationship between the number of solutions of congruence systems and the resultant of two polynomials, \it INTEGERS -- Electronic Journal of Combinatorial Number Theory \rm {\bf 16}, A41

\end{thebibliography}
\end{document}